\documentclass[11pt,a4paper]{article}
\usepackage[T1]{fontenc}
\usepackage[utf8]{inputenc}
\usepackage[margin=1in]{geometry}
\usepackage{amsmath,amssymb,amsthm}
\usepackage{graphicx}
\usepackage{enumitem}
\usepackage{booktabs}
\usepackage{multirow}
\usepackage{threeparttable}
\graphicspath{{fig/}}
\usepackage{url}
\usepackage[hidelinks]{hyperref}
\newtheorem{theorem}{Theorem}[section]
\newtheorem{lemma}[theorem]{Lemma}
\newtheorem{conjecture}[theorem]{Conjecture}
\newtheorem{corollary}[theorem]{Corollary}
\newtheorem{proposition}[theorem]{Proposition}
\newtheorem{definition}[theorem]{Definition}
\newtheorem{remark}[theorem]{Remark}
\newtheorem{problem}[theorem]{Problem}

\newcommand{\tr}{\operatorname{tr}}

\newcommand{\ind}{\operatorname{ind}}
\newcommand{\Forb}{\operatorname{Forb}}

\title{Generalized spectral closedness of $\mathcal{F}$-free graph classes}
\author
{
	Wei Wang$^{\rm a}$\thanks{Corresponding author.\\
	E-mail address:  wangwei.math@gmail.com (Wei Wang), tangquanyu827@gmail.com (Quanyu Tang)}\quad\quad
	Quanyu Tang$^{\rm b}$
	\\
	{\footnotesize$^{\rm a}$School of Mathematics, Physics and Finance, Anhui Polytechnic University, Wuhu 241000, P. R. China}\\
	{\footnotesize$^{\rm b}$School of Mathematics and Statistics, Xi'an Jiaotong University, Xi'an 710049, P. R. China}
}
\date{}

\begin{document}
	
\maketitle
	
\begin{abstract}
In this paper, we investigate the generalized spectral closedness of graph classes defined by a family $\mathcal{F}$ of forbidden induced subgraphs. To systematically study this property, we introduce a novel combinatorial concept of patterned closed walks (or $\beta$-closed walks), which naturally interlaces the edges of a graph with those of its complement. By establishing the induced-subgraph expansion of these $\beta$-closed walk counts, we obtain an algebraic sufficient condition for generalized spectral closedness based on the existence of a walk-realizable $\mathcal{F}$-supporter. Crucially, the search for such a walk-realizable supporter is reduced to a linear programming feasibility problem. As primary applications of this computational framework, we prove that the classes of threshold graphs and chain graphs are generalized spectrally closed.
\end{abstract}
	
	

\section{Introduction}

In spectral graph theory, the \emph{spectrum} of a graph $G$ refers to the multiset of eigenvalues of its adjacency matrix $A(G)$. To capture richer structural information, the concept of the \emph{generalized spectrum} is frequently considered, which refers to the ordinary spectrum of $G$ together with the spectrum of its complement, $\overline{G}$. Two graphs $G$ and $H$ are \emph{generalized cospectral} if they share the same generalized spectrum. A graph $G$ is \emph{determined by its generalized spectrum} (DGS) if any graph generalized cospectral with $G$ is isomorphic to $G$. Moving beyond individual graphs to families of graphs, for a given graph class $\mathcal{C}$, we say $\mathcal{C}$ is \emph{generalized spectrally closed} if no graph in $\mathcal{C}$ can be generalized cospectral with a graph outside $\mathcal{C}$.

In recent years, the DGS property of individual graphs has been extensively studied. For instance, Wang \cite{wang2017} proposed a simple arithmetic criterion for a graph to be DGS. This criterion was recently improved by Wang, Wang, and Zhu \cite{wang2023}, and further generalized by Guo, Wang, and Wang \cite{guo2025} using the primary decomposition theorem. 

While the generalized spectral characterization of individual graphs is well-explored, understanding the generalized spectral closedness of specific graph classes remains a fundamental and natural problem. Traditionally, significant attention has been given to graph classes defined by global properties, such as having a perfect matching, being Hamiltonian, or having a specific chromatic index or connectivity. Clarifying whether these important families are spectrally closed provides deep insights into the separating power of graph spectra. Interestingly, a growing body of literature indicates that these global properties are often not generalized spectrally closed, as demonstrated by various constructions of regular cospectral counterexamples \cite{blazsik2015, etesami2020, yan2022, haemers2020, liu2020}.

Alongside global properties, local structural constraints offer another prominent and fundamental way to define graph classes, specifically hereditary graph classes characterized by forbidden induced subgraphs. Let $\mathcal{F}$ be a family of graphs. A graph $G$ is said to be \emph{$\mathcal{F}$-free} if it contains no induced subgraph isomorphic to any graph in $\mathcal{F}$; the class of $\mathcal{F}$-free graphs is denoted by $\Forb(\mathcal{F})$. Many well-known graph classes can be elegantly characterized in this manner. For example, if $\mathcal{F} = \{C_3, C_5, \ldots\}$ is the collection of all odd cycles, then $\Forb(\mathcal{F})$ is the class of bipartite graphs. Similarly, the classes of threshold graphs and cographs are exactly $\Forb(2K_2, P_4, C_4)$ and $\Forb(P_4)$, respectively. See Table \ref{fc} for more examples.

\begin{table}[htbp]
	\centering
	\caption{Some graph classes defined by forbidden induced subgraphs.} 
	\label{fc} 
	\begin{threeparttable}   
		\begin{tabular}{@{}ll@{}}     
			\toprule      
			\textbf{Graph class}  & \textbf{Forbidden induced subgraphs} \\  
			\midrule   
			Cluster graphs \cite{brandstadt1999}       & $P_3$ \\  
			Claw-free graphs \cite{faudree1997}      & $K_{1,3}$ \\
			Cographs \cite{corneil1981}              & $P_4$ \\
			Bull-free graphs \cite{chvatal1987}      & the bull graph\tnote{a} \\ 
			Trivially perfect graphs \cite{golumbic1978,yan1996} & $P_4,\; C_4$ \\
			Threshold graphs \cite{chvatal1977}      & $2K_2,\; P_4,\; C_4$ \\
			Split graphs \cite{foldes1977}         & $2K_2,\; C_4,\; C_5$ \\ 
			Chain graphs \cite{yannakakis1981,hammer1990}         & $2K_2,\; C_3,\; C_5$ \\
			\bottomrule  
		\end{tabular}  
		\begin{tablenotes}
			\small
			\item[a] The bull graph is a 5-vertex graph consisting of a triangle and two disjoint pendant edges (e.g., $H_{25}$ in Table~\ref{smallgraphs}).
		\end{tablenotes}  
	\end{threeparttable}
\end{table}

Given the structural importance of hereditary graph classes, investigating their generalized spectral closedness is just as compelling as studying global properties. Clarifying whether these $\Forb(\mathcal{F})$ classes are generalized spectrally closed is crucial for mapping the boundaries of spectral characterization and deepening our understanding of what structural information the spectrum actually captures. However,  research in this direction remains largely absent. As an exploratory attempt, this paper aims to provide a systematic and algebraic framework to test the generalized spectral closedness of such hereditary graph classes.

The primary objective of this paper is to partially address the following problem.

\begin{problem}
	For which finite families of graphs $\mathcal{F}$ is the graph class $\Forb(\mathcal{F})$ generalized spectrally closed?
\end{problem}

Naturally, a similar problem can be proposed in the context of the ordinary spectrum instead of the generalized spectrum. A well-known result in spectral graph theory states that $\Forb(C_3)$, the class of triangle-free graphs, is spectrally closed. Indeed, the number of triangles in a graph $G$ with adjacency matrix $A$ is equal to $\frac{1}{6}\tr (A^3)$ and hence is determined by its spectrum. Thus, any graph cospectral with a triangle-free graph must also be triangle-free, that is, the class $\Forb(C_3)$ is spectrally closed. We note by definition that the property of spectral closedness is stronger than that of generalized spectral closedness. In other words, a spectrally closed class is necessarily generalized spectrally closed, but the converse may not be true in general; see Remark \ref{cnc}.

Studies of the generalized spectral closedness of graph classes may have potential use in the problem of generalized spectral characterizations of individual graphs. Precisely, in order to show each graph in a given graph class $\mathcal{C}$ is DGS, it suffices to show that the following two assertions hold simultaneously:
\begin{itemize}
	\item  no two (non-isomorphic) graphs in $\mathcal{C}$ are generalized cospectral;
	\item the class $\mathcal{C}$ is generalized spectrally closed.
\end{itemize}

Let $\mathcal{C}=\Forb(2K_2,P_4,C_4)$ be the class of threshold graphs. A beautiful result of Lazzarin et al. \cite{lazzarin2019} states that no two graphs in $\mathcal{C}$ are cospectral. Thus, if we can show that the class $\mathcal{C}$ is spectrally closed, then each graph in $\mathcal{C}$ is determined by its spectrum (DS). Recall that the number of threshold graphs of order $n$ is exactly $2^{n-1}$. Thus, assuming the spectral closedness of $\mathcal{C}$, the number of DS graphs of order $n$ is at least $2^{n-1}$, which would improve the recent lower bound $(1+\epsilon)^n$ due to Koval and Kwan \cite{koval2024}. The spectral closedness of the class of threshold graphs is currently unknown. However, as an important example of our main results, we prove that the corresponding weaker version for the generalized spectrum turns out to be true.

\begin{theorem}\label{tgsc}
	The class of threshold graphs is generalized spectrally closed.
\end{theorem}

Noting that generalized cospectrality is a stricter requirement than ordinary cospectrality, the result of Lazzarin et al. \cite{lazzarin2019} clearly implies that no two threshold graphs are generalized cospectral. This, together with Theorem \ref{tgsc}, indicates that each threshold graph is DGS.

Chain graphs, which are often viewed as the bipartite analogue of threshold graphs, constitute another fundamental graph class defined by forbidden induced subgraphs. Similar to threshold graphs, we establish their generalized spectral closedness.

\begin{theorem}\label{cgsc}
	The class of chain graphs is generalized spectrally closed.
\end{theorem}
\begin{remark}\normalfont
Theorem \ref{cgsc} does not imply that each chain graph is DGS. Indeed, in contrast to the aforementioned result of Lazzarin et al.~\cite{lazzarin2019}, An\dj{}eli\'c and Tura \cite{andelic2025} showed that there exist many cospectral chain graphs. Moreover, using the standard  binary sequence representation of chain graphs as in \cite{andelic2025}, generalized cospectral chain graphs are easy to find through an exhaustive search among small chain graphs. For example, the two chain graphs generated by $(0^21^3)(0^11^1)(0^21^1)$ and $(0^21^1)(0^11^2)(0^31^1)$ are generalized cospectral.
\end{remark}
\begin{remark}\normalfont
	\label{cnc}
	Fig.~\ref{chain} gives a pair of cospectral bipartite graphs, where one is a chain graph and the other is not.  This indicates that the class of chain graphs is not spectrally closed.
\end{remark}
	\begin{figure}
	\centering
	\includegraphics[height=3.5cm]{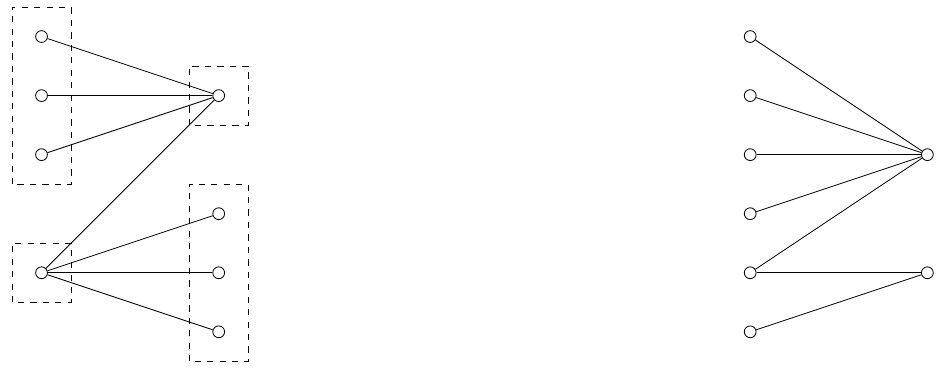}
	\caption{A chain graph (left) and its  cospectral mate that is not a chain graph (right).}
	\label{chain}
\end{figure}
The proofs of Theorem \ref{tgsc} and Theorem \ref{cgsc} are not merely isolated combinatorial arguments; rather, they are specific applications of a broader theoretical framework developed in this paper. To systematically study the generalized spectral closedness of an $\mathcal{F}$-free graph class, we first introduce a purely combinatorial notion of an $\mathcal{F}$-supporter, which detects membership in $\Forb(\mathcal{F})$ through induced-subgraph counts. Since such a supporter has no spectral meaning in general, the crucial step is to determine whether it can be realized as a linear combination of $\beta$-closed walk counts. We call such a certificate a \emph{walk-realizable $\mathcal{F}$-supporter}, and formulate the search for it as a linear programming feasibility problem.  The existence of such a supporter provides a
verifiable algebraic sufficient condition for generalized spectral closedness, yielding an
effective automated framework for certifying several important $\mathcal{F}$-free graph
classes.

The remainder of this paper is organized as follows. In Section 2, we formalize the concept of patterned closed walks associated with a binary vector $\beta$ (or $\beta$-closed walk),  which generalizes the ordinary closed walk in a graph and uses edges and non-edges in each step according to the entry of $\beta$. We show that the total number of $\beta$-closed walks of a graph $G$, denoted $n_{\beta}(G)$,  is an invariant under generalized cospectrality. In Section 3, we further  obtain the induced subgraph expansion of the number $n_{\beta}(G)$, using the principle of inclusion--exclusion. In Section~4, we define $\mathcal{F}$-supporters and their walk-realizable variants,
prove that the existence of a walk-realizable $\mathcal{F}$-supporter guarantees
generalized spectral closedness, and apply this computational framework to analyze
several well-known graph classes listed in Table \ref{fc}. 

\section{\texorpdfstring{$\beta$-closed walks}{beta-closed walks} and generalized spectral invariance}

Let $\beta=(\beta_1,\beta_2,\dots,\beta_k)\in\{0,1\}^k$ be a binary vector of length $k$. A \emph{patterned closed walk associated with $\beta$} in a graph $G$ is a sequence of vertices $v_1, v_2, \dots, v_{k+1}$ such that $v_{k+1}=v_1$, and for each $i \in \{1,2,\dots,k\}$:
\begin{itemize}
	\item if $\beta_i = 1$, $v_i v_{i+1}$ is an edge of $G$;
	\item if $\beta_i = 0$, $v_i v_{i+1}$ is an edge of the complement graph $\overline{G}$.
\end{itemize}

For brevity, we refer to such a walk as a \emph{$\beta$-closed walk}. Note that if $\beta$ is the all-ones vector, a $\beta$-closed walk in $G$ reduces to a standard closed walk in $G$. Conversely, if $\beta$ is the zero vector, a $\beta$-closed walk corresponds to an ordinary closed walk in $\overline{G}$. 

It is a well-known result that the total number of standard closed walks of length $k$ in a graph equals the trace of the $k$-th power of its adjacency matrix \cite{cvetkovic2010}. Generalizing this classical connection, the total number of $\beta$-closed walks can be naturally expressed via the trace of an ordered product of adjacency matrices. To formalize this, for any graph $G$, we define the indicator matrices:
\begin{equation}
	\label{eq:indicator-matrices}
	M_1(G) = A(G), \qquad M_0(G) = A(\overline{G}).
\end{equation}
Consequently, with the notation in \eqref{eq:indicator-matrices}, letting
$n_\beta(G)$ denote the total number of $\beta$-closed walks in $G$, we obtain
\begin{equation}
	\label{eq:beta-trace}
	n_\beta(G) = \operatorname{tr}\left(\prod_{i=1}^k M_{\beta_i}(G)\right).
\end{equation}

To prove that the number of $\beta$-closed walks is a generalized spectral invariant for any $\beta$, we shall employ the following matrix characterization of generalized cospectrality of graphs established by Johnson and Newman \cite{johnson1980}. An orthogonal  matrix $Q$ is  \emph{regular} if each row sum of $Q$ is 1. 
\begin{theorem}[\cite{johnson1980}]
	\label{js}
	Let $G$ and $H$ be graphs on the same number of vertices, with adjacency matrices $A(G)$ and $A(H)$. Then $G$ and $H$ are generalized cospectral if and only if there exists a regular orthogonal matrix $Q$ such that
	\begin{equation}\label{qa}
    Q^\top A(G) Q = A(H).
	\end{equation}
\end{theorem}

Building upon this theorem, the following lemma formally establishes that the $\beta$-closed walk counts are indeed generalized spectral invariants, demonstrating that these numbers $n_{\beta}(G)$ are strictly conserved under generalized cospectrality.

\begin{lemma}
	\label{binv}
	If $G$ and $H$ are two graphs of order $n$, then  the following statements are equivalent:
	
	(i) $G$ and $H$ are generalized cospectral;
	
	(ii)  $n_\beta(G) = n_\beta(H)$
	for every integer $k \ge 1$ and every binary vector $\beta \in \{0,1\}^k$;
	
	(iii) $n_\beta(G)= n_\beta(H)$ for $\beta\in\{\mathbf{1}_k\colon\, 1\le  k\le n\}\cup\{\mathbf{0}_k\colon\, 1\le k\le n\}$.
\end{lemma}

\begin{proof}
	We prove the lemma by establishing the cyclic implication (i)$\implies$(ii)$\implies$(iii)$\implies$(i). Suppose (i) holds. Then, by Theorem \ref{js},  there exists an orthogonal matrix $Q$ satisfying \eqref{qa}. Since $Q$ is orthogonal and $Q\mathbf{1} = \mathbf{1}$, it naturally follows that $Q^\top\mathbf{1} = \mathbf{1}$. Let $J$ denote the all-ones matrix and $I$ the identity matrix. We then have $Q^\top J Q = J$. 
	
	Consequently, the adjacency matrix of the complement graph $\overline{G}$ transforms as follows:
	\begin{align*}
		Q^\top A(\overline{G}) Q &= Q^\top (J - I - A(G)) Q \\
		&= Q^\top J Q - Q^\top I Q - Q^\top A(G) Q \\
		&= J - I - A(H) \\
		&= A(\overline{H}).
	\end{align*}
	Thus, for any indicator value $\delta \in \{0,1\}$, the corresponding matrices are simultaneously conjugated by $Q$: $Q^\top M_\delta(G) Q = M_\delta(H)$. Substituting this into the trace formulation \eqref{eq:beta-trace}, the internal $Q Q^\top$ terms cancel out due to orthogonality, yielding:
	\[
	n_\beta(G) = \operatorname{tr}\left(\prod_{i=1}^k Q M_{\beta_i}(H) Q^\top\right) = \operatorname{tr}\left(Q \left(\prod_{i=1}^k M_{\beta_i}(H)\right) Q^\top\right) = n_\beta(H).
	\]
	This proves (i)$\implies$(ii). The implication (ii)$\implies$(iii) is clear. It remains to show (iii)$\implies$(i). Indeed, noting that $n_{\mathbf{1}_k}(G)=\tr(A(G)^k)$ and $n_{\mathbf{0}_k}(G)=\tr(A(\overline{G})^k)$, the equality of these power sums for $1\le k\le n$ completely determines the eigenvalues of both $A(G)$ and $A(\overline{G})$ via Newton's identities. Consequently, $G$ and $H$ must share the same generalized spectrum. Thus (iii)$\implies$(i), completing the proof of Lemma \ref{binv}.
\end{proof}

\section{The induced-subgraph expansion}

For any binary vector $\beta = (\beta_1, \beta_2, \dots, \beta_k) \in \{0,1\}^k$, we have 
$$n_\beta(G) = \operatorname{tr}\left( \prod_{i=1}^k M_{\beta_i}(G) \right).$$
Although there are $2^k$ possible binary sequences of length $k$, many of them yield identical trace evaluations. This redundancy is a direct consequence of the algebraic symmetries of the trace operator under cyclic shift and reversal:

\begin{enumerate}
	\item \textbf{Cyclic Invariance (Rotations):} For any matrices $X_1, \dots, X_k$, the cyclic property $\operatorname{tr}(X_1 X_2 \cdots X_k) = \operatorname{tr}(X_2 \cdots X_k X_1)$ implies that a cyclic shift of the binary sequence $\beta$ does not alter $n_\beta(G)$.
	\item \textbf{Reversal Invariance (Reflections):} Since $G$ is undirected, each indicator matrix is symmetric ($M_{\beta_i}^T = M_{\beta_i}$). Applying the transpose-trace identity $\operatorname{tr}(W) = \operatorname{tr}(W^T)$, we have:
	$$\operatorname{tr}(M_{\beta_1} M_{\beta_2} \cdots M_{\beta_k}) = \operatorname{tr}(M_{\beta_k} \cdots M_{\beta_2} M_{\beta_1}),$$
	which implies that reversing the sequence $\beta$ preserves $n_\beta(G)$.
\end{enumerate}

These two algebraic symmetries naturally define an action of the dihedral group $D_k$ (generated by rotations and reflections) on the set of binary sequences $\{0,1\}^k$. Sequences lying in the same $D_k$-orbit yield the same closed-walk count for every graph $G$.

Geometrically, each orbit under this $D_k$-action corresponds to a unique 2-colored \textit{bracelet} with $k$ beads, where the two colors represent the binary values $1$ and $0$ of the sequence. Thus, when evaluating the induced-subgraph expansions of length $k$, it is sufficient to choose one representative sequence from each bracelet orbit. This completely removes the computational redundancy coming from the cyclic and reversal properties of the trace. By applying P\'olya's Enumeration Theorem (see, e.g., \cite[p.~525]{vanLint2001}), we obtain the exact number of such bracelet orbits.

\begin{proposition}
	The number of binary sequences of length $k$, up to cyclic shifts and reversal, is equal to the number of $2$-colored bracelets of length $k$, namely
	$$N(k) = \frac{1}{2k} \left( \sum_{d|k} \varphi(d) 2^{k/d} + 
	\begin{cases} 
		k \cdot 2^{\frac{k+1}{2}}, & \text{if } k \text{ is odd,} \\
		\frac{k}{2} \cdot 2^{\frac{k+2}{2}} + \frac{k}{2} \cdot 2^{\frac{k}{2}}, & \text{if } k \text{ is even,}
	\end{cases}
	\right)$$
	where $\varphi(d)$ denotes Euler's totient function.
\end{proposition}

This proposition provides a drastic reduction in the search space of matrix products that must be considered. Instead of evaluating all $2^k$ possible binary sequences, we only need to evaluate one canonical representative from each bracelet orbit. To bridge the geometric concept of bracelets with the algebraic computation of $n_\beta(G)$, we can naturally represent each binary sequence $\beta$ by its corresponding matrix product, applying the mapping $1 \mapsto A$ and $0 \mapsto \overline{A}$. For example, the bracelet representative $\beta = (1, 1, 0)$ corresponds to the matrix product $A^2\overline{A}$. Canonical representatives of the bracelet orbits for $k\le 5$, expressed in this intuitive product notation, are listed in Table~\ref{bra}.

\begin{table}[htbp]
	\centering
	\caption{List of bracelet orbit representatives for length $k \le 5$, expressed as matrix products.}
	\label{bra}
	\renewcommand{\arraystretch}{1.2}
	\begin{tabular}{@{}ccl@{}}
		\toprule
		Length ($k$) & Number of orbits $N(k)$ & Representatives (Matrix Products) \\
		\midrule
		$1$ & $2$ & $A, \overline{A}$ \\
		$2$ & $3$ & $A^2, A\overline{A}, \overline{A}^2$ \\
		$3$ & $4$ & $A^3, A^2\overline{A}, A\overline{A}^2, \overline{A}^3$ \\
		$4$ & $6$ & $A^4, A^3\overline{A}, A^2\overline{A}^2, A\overline{A}A\overline{A}, A\overline{A}^3, \overline{A}^4$ \\
		$5$ & $8$ & $A^5, A^4\overline{A}, A^3\overline{A}^2, A^2\overline{A}A\overline{A}, A^2\overline{A}^3, A\overline{A}A\overline{A}^2, A\overline{A}^4, \overline{A}^5$ \\
		\bottomrule
	\end{tabular}
\end{table}

Let $\sim$ denote the equivalence relation on $\{0,1\}^k$ induced by the dihedral group action. The set of equivalence classes (the $k$-bracelets) is denoted as $\{0,1\}^k / D_k$. For algebraic formulation and computational efficiency, we define $\mathcal{B}_k \subset \{0,1\}^k$ to be a fixed set of canonical representatives for these equivalence classes. Furthermore, we denote the collection of all such representatives up to length $\ell$ as $\mathcal{B}_{\le \ell} = \bigcup_{k=1}^{\ell} \mathcal{B}_k$.

While bracelets are elegant algebraic invariants, their true power lies in their combinatorial counterpart. Every bracelet evaluates to a precise linear combination of induced subgraph counts. 

For a graph $H$ and a vector $\beta \in \{0,1\}^k$, we define $s_\beta(H)$ to be the number of \emph{surjective} $\beta$-closed walks on $H$, that is, $\beta$-closed walks $v_1, v_2, \dots, v_{k+1}$ in $H$ such that the set of visited vertices $\{v_1, v_2, \dots, v_k\}$ is exactly $V(H)$.
For an integer $\ell\ge 1$, let $\mathcal{U}_{\le \ell}$ denote the set of isomorphism classes of nonempty graphs with at most $\ell$ vertices. For any graphs $H$ and $G$, let $\ind(H, G)$ denote the number of induced subgraphs of $G$ that are isomorphic to $H$. Let $n_\beta(G)$ denote the total number of $\beta$-closed walks in a graph $G$. As established, $n_\beta(G) = \operatorname{tr}\left(\prod_{i=1}^k M_{\beta_i}\right)$. 

	With these definitions in place, we can now formulate the precise connection between $\beta$-closed walks and induced subgraphs. The following lemma formalizes the induced-subgraph expansion of $n_\beta(G)$ and provides an explicit inclusion--exclusion formula to compute its corresponding coefficients using matrix traces.

\begin{lemma}	\label{ide}
	For every vector $\beta \in \{0,1\}^k$, the total number of $\beta$-closed walks in $G$ can be expanded as
	\begin{equation}
		\label{nid}
		n_\beta(G) = \sum_{H\in\mathcal{U}_{\le k}} s_\beta(H)\ind(H,G),
	\end{equation}
	where 
	\begin{equation}
		\label{sbh}
		s_\beta(H) = \sum_{L \in \mathcal{U}_{\le |V(H)|}} (-1)^{|V(H)|-|V(L)|} \operatorname{tr}\left(\prod_{i=1}^k M_{\beta_i}(L)\right) \ind(L, H),
	\end{equation}
	with $M_{\beta_i}(L)$ denoting the respective adjacency matrix (either $A(L)$ or $A(\overline{L})$) dictated by $\beta_i$.
\end{lemma}

\begin{proof}
	Every $\beta$-closed walk $v_1, v_2, \dots, v_{k+1}$ (where $v_1=v_{k+1}$) in $G$ visits a specific subset of vertices $S = \{v_1, \dots, v_k\} \subseteq V(G)$. Since the walk has length $k$, we clearly have $|S| \le k$. The induced subgraph $G[S]$ belongs to some isomorphism class $H \in \mathcal{U}_{\le k}$. Grouping all $\beta$-closed walks in $G$ by the isomorphism class of the induced subgraph spanned by their vertices, we see that each induced subgraph isomorphic to $H$ contributes exactly $s_\beta(H)$ walks (since the walk must visit all vertices of $G[S]$, corresponding to a surjective $\beta$-closed walk). Summing over all isomorphism classes $H \in \mathcal{U}_{\le k}$ yields \eqref{nid}.
	
	To compute $s_\beta(H)$ for a fixed graph $H$, we apply the principle of inclusion--exclusion. For a subset $X \subseteq V(H)$, the total number of $\beta$-closed walks strictly confined within the induced subgraph $H[X]$ is given by the trace of its corresponding matrix product: $\operatorname{tr}\left(\prod_{i=1}^k M_{\beta_i}(H[X])\right)$. Applying the standard inclusion--exclusion formula over all subsets $X \subseteq V(H)$ to isolate the walks that span exactly $V(H)$, we obtain
	\[
	s_\beta(H) = \sum_{X\subseteq V(H)}(-1)^{|V(H)|-|X|} \operatorname{tr}\left(\prod_{i=1}^k M_{\beta_i}(H[X])\right).
	\]
	Instead of summing over all subsets $X$, we group the terms by the isomorphism class $L$ of the induced subgraph $H[X]$. The number of subsets $X \subseteq V(H)$ such that $H[X] \cong L$ is precisely $\ind(L,H)$. For each such subset, $|X| = |V(L)|$. Furthermore, since isomorphic graphs yield similar adjacency matrices, the trace is invariant under isomorphism, meaning $\tr\left(\prod_{i=1}^k M_{\beta_i}(H[X])\right) = \tr\left(\prod_{i=1}^k M_{\beta_i}(L)\right)$. Substituting these relations into the sum simplifies the expression into \eqref{sbh}, completing the proof.
\end{proof}
For a binary vector $\beta\in\{0,1\}^k$, write
$\overline{\beta}=\mathbf{1}_k-\beta$. We shall use the following simple
complementation symmetry repeatedly.
\begin{lemma}\label{sym}
For any graph $G$ and binary vector $\beta$, we have 

(i) $n_{\overline{\beta}}(G)=n_{\beta}(\overline{G})$;

(ii) $s_{\overline{\beta}}(G)=s_{\beta}(\overline{G})$;

(iii) if the induced-subgraph expansion of $n_\beta(G)$ is
$ n_\beta(G) = \sum_{H\in\mathcal{U}_{\le k}} c_H\ind(H,G),
$
then $
	n_{\overline{\beta}}(G) = \sum_{H\in\mathcal{U}_{\le k}} c_H\ind(\overline{H},G).
$
\end{lemma}
\begin{proof}
	(i) By definition, the $i$-th step of a $\overline{\beta}$-closed walk in $G$ traverses an edge of $G$ when $\beta_i=0$, and an edge of $\overline{G}$ when $\beta_i=1$. This precisely matches the adjacency condition for the $i$-th step of a $\beta$-closed walk in $\overline{G}$. Therefore, the two sets of walks are strictly identical, yielding $n_{\overline{\beta}}(G)=n_{\beta}(\overline{G})$.
	
	(ii) The trivial identity bijection established in (i) perfectly preserves the set of visited vertices. Hence, a sequence of vertices forms a surjective $\overline{\beta}$-closed walk on $G$ if and only if the exact same sequence forms a surjective $\beta$-closed walk on $\overline{G}$. Thus, $s_{\overline{\beta}}(G) = s_{\beta}(\overline{G})$.
	
	(iii) Evaluating the induced-subgraph expansion on $\overline{G}$ gives 
	$n_{\beta}(\overline{G}) = \sum_{H\in\mathcal{U}_{\le k}} c_H\ind(H,\overline{G})$. 
	Substituting $n_{\beta}(\overline{G}) = n_{\overline{\beta}}(G)$ from (i) and the standard structural identity $\ind(H, \overline{G}) = \ind(\overline{H}, G)$ immediately yields the desired result.
\end{proof}
Equipped with the explicit inclusion--exclusion formula from Lemma \ref{ide} and the combinatorial symmetries established in Lemma \ref{sym}, we are now positioned to fully expand the canonical bracelets identified in Table \ref{bra}. 

Crucially, Lemma \ref{sym}(iii) halves our computational burden. It guarantees that we only need to compute the induced-subgraph expansion for one representative of each complementary pair of binary vectors. For instance, calculating the expansion for $\beta = (1,1,0,1,0)$ immediately yields the expansion for its complement $\overline{\beta} = (0,0,1,0,1)$ by simply taking the graph complements of the resulting subgraphs. The complete combinatorial evaluation of these independent orbits for walks up to length 5 is summarized in the following theorem.

\begin{theorem}\label{expansion}
	For every canonical binary vector $\beta$ of length $1\le k\le 5$, up to
	the dihedral symmetries and the complementation symmetry of
	Lemma~\ref{sym}, the total number of $\beta$-closed walks in a graph $G$
	admits the induced-subgraph expansion displayed in Table~\ref{efsb}.	
\end{theorem}

	\begin{table}[htpb]
		\centering
		\begin{threeparttable}
			\caption{Induced subgraph expansion of canonical bracelets ($k \le 5$)}
			\label{efsb}
			\renewcommand{\arraystretch}{1.5}
			\begin{tabular}{@{}cc p{0.65\textwidth} @{}}
				\toprule
				$k$ &  $\beta\in\{0,1\}^k$ & Expansion Formula $n_\beta(G)$ \\
				\midrule
				$1$ & $(1)$ & $0$ \\
				\midrule
				\multirow{2}{*}{2} & $(1,1)$ & $2H_{2}$ \\
				& $(1,0)^*$ & $0$ \\
				\midrule
				\multirow{2}{*}{3} & $(1,1,1)$ & $6H_{3}$ \\
				& $(1,1,0)$ & $2H_{4}$ \\
				\midrule
				\multirow{4}{*}{4} & $(1,1,1,1)$ & $2H_{2} + 12H_{3} + 4H_{4} + 24H_{5} + 8H_{6} + 8H_{9}$ \\
				& $(1,1,1,0)$ & $4H_{6} + 4H_{7} + 2H_{10}$ \\
				& $(1,1,0,0)^*$ & $2H_{4} + 2\overline{H_{4}} + 2H_{7} + 2\overline{H_{7}} + 6H_{8} + 6\overline{H_{8}}$ \\
				& $(1,0,1,0)^*$ & $8H_{9} + 8\overline{H_{9}} + 4H_{10}$ \\
				\midrule
				\multirow{4}{*}{5} & $(1,1,1,1,1)$ & $30H_{3} + 120H_{5} + 40H_{6} + 10H_{7} + 120H_{11} + 60H_{12} + 20H_{13} + 40H_{16} + 10H_{17} + 20H_{18} + 10H_{27} + 10H_{28}$ \\
				& $(1,1,1,1,0)$ & $4H_{4} + 12H_{6} + 8H_{7} + 6H_{8} + 16H_{9} + 4H_{10} + 12H_{12} + 16H_{13} + 12H_{14} + 8H_{16} + 10H_{17} + 8H_{18} + 12H_{19} + 8H_{20} + 4H_{21} + 8H_{23} + 4H_{24} + 4\overline{H_{24}} + 2H_{25} + 12H_{26} + 4H_{27} + 2\overline{H_{27}}$ \\
				& $(1,1,1,0,0)$ & $2\overline{H_{4}} + 4H_{6} + 8H_{7} + 4\overline{H_{7}} + 18\overline{H_{8}} + 4H_{10} + 4H_{13} + 12H_{14} + 24\overline{H_{15}} + 4H_{17} + 2\overline{H_{17}} + 12H_{19} + 4H_{20} + 4\overline{H_{20}} + 12H_{21} + 4\overline{H_{21}} + 8H_{22} + 12\overline{H_{22}} + 8\overline{H_{23}} + 4\overline{H_{24}} + 6H_{25} + 2H_{27}$ \\
				& $(1,1,0,1,0)$ & $2H_{7} + 8\overline{H_{9}} + 4H_{10} + 8H_{16} + 6H_{17} + 8H_{18} + 4\overline{H_{18}} + 4H_{20} + 4\overline{H_{20}} + 4H_{21} + 4H_{22} + 8H_{23} + 8H_{24} + 4\overline{H_{24}} + 4H_{25} + 12\overline{H_{26}} + 10H_{27} + 4\overline{H_{27}} + 10H_{28}$ \\
				\bottomrule
			\end{tabular}
			\begin{tablenotes}
				\item[\textsuperscript{a}] For brevity, $H_i$ and $\overline{H_i}$ in the
				expansion formulas denote $\operatorname{ind}(H_i,G)$ and
				$\operatorname{ind}(\overline{H_i},G)$, respectively, where the graphs
				$H_i$ are labeled as in Table~\ref{smallgraphs}.
				\item[\textsuperscript{b}] The asterisk $({}^*)$ indicates that
				$\mathbf{1}_k-\beta$ belongs to the same dihedral orbit as $\beta$;
				hence $n_{\beta}(G)=n_{\mathbf{1}_k-\beta}(G)$ for that listed vector.
			\end{tablenotes}
		\end{threeparttable}
	\end{table}
	\begin{table}[htbp]
		\centering
		\caption{All graphs (up to complementation) of order $\le 5$, labeled $H_1$ through $H_{28}$.}
		\label{smallgraphs}
		\setlength{\tabcolsep}{0.5pt}
		\renewcommand{\arraystretch}{0.5}
		\begin{tabular}{ccccccc}
			\includegraphics[width=0.13\textwidth]{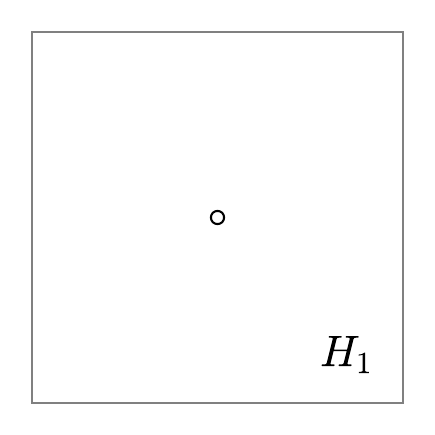} & 
			\includegraphics[width=0.13\textwidth]{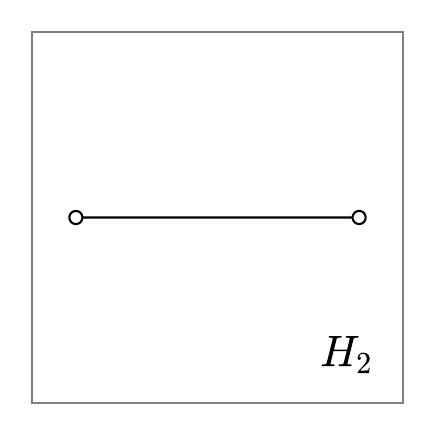} & 
			\includegraphics[width=0.13\textwidth]{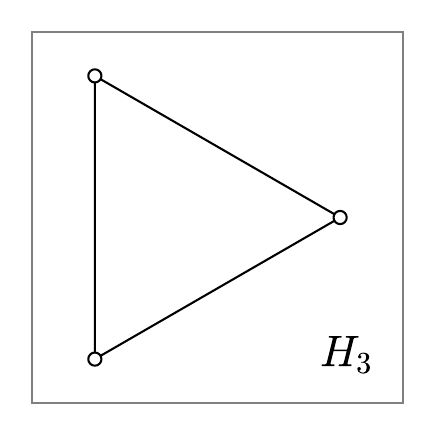} & 
			\includegraphics[width=0.13\textwidth]{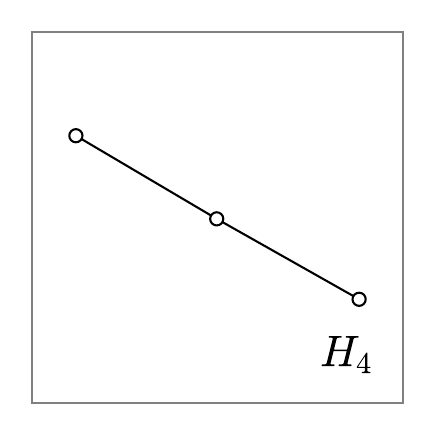} & 
			\includegraphics[width=0.13\textwidth]{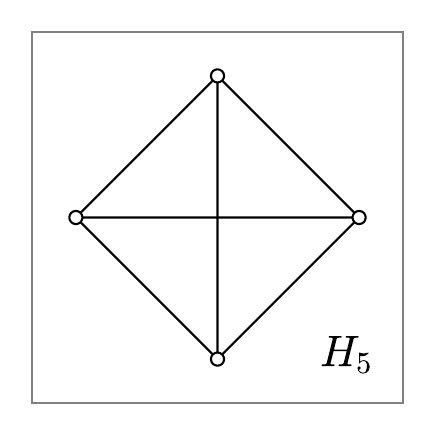} & 
			\includegraphics[width=0.13\textwidth]{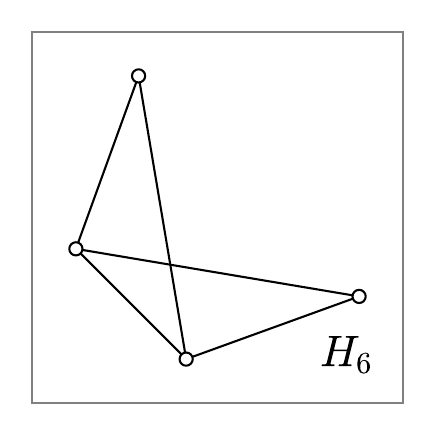} & 
			\includegraphics[width=0.13\textwidth]{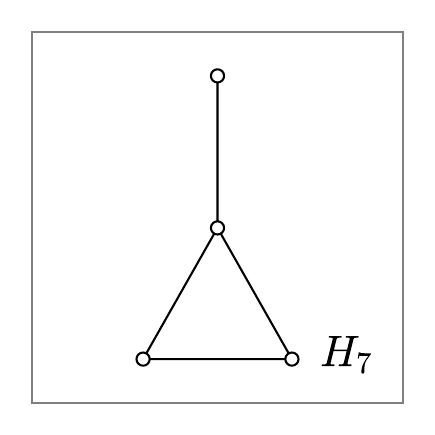} \\
			\includegraphics[width=0.13\textwidth]{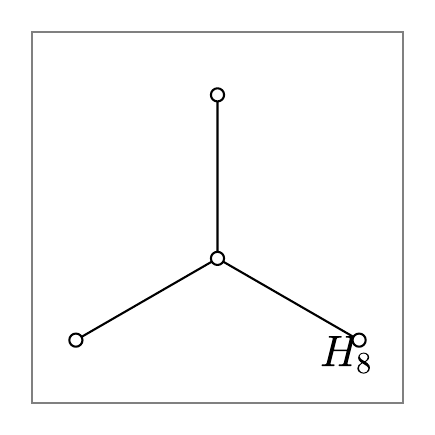} & 
			\includegraphics[width=0.13\textwidth]{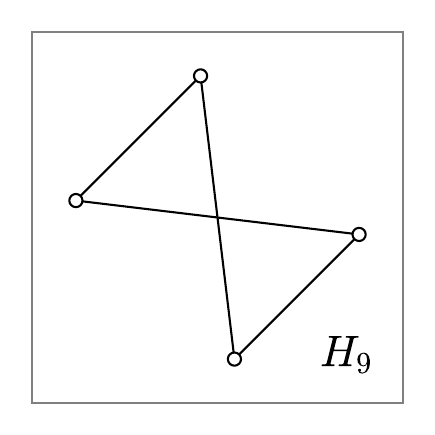} & 
			\includegraphics[width=0.13\textwidth]{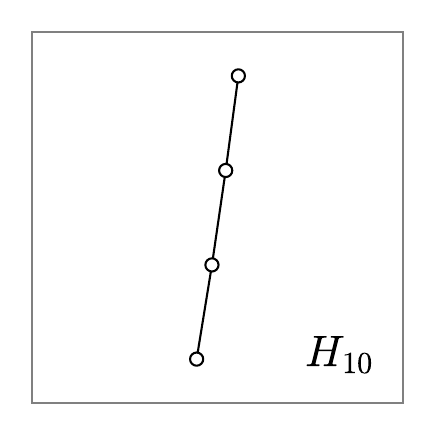} & 
			\includegraphics[width=0.13\textwidth]{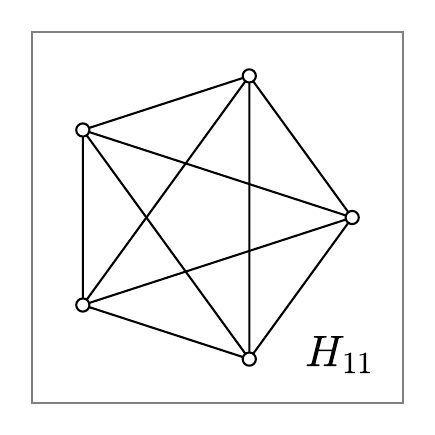} & 
			\includegraphics[width=0.13\textwidth]{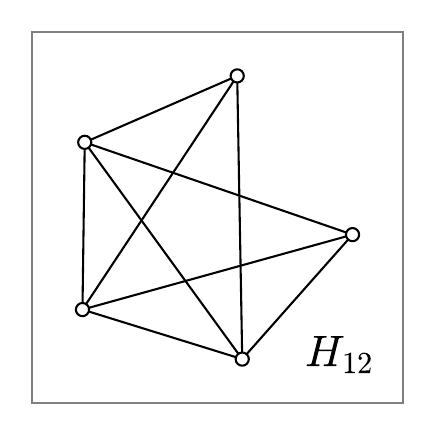} & 
			\includegraphics[width=0.13\textwidth]{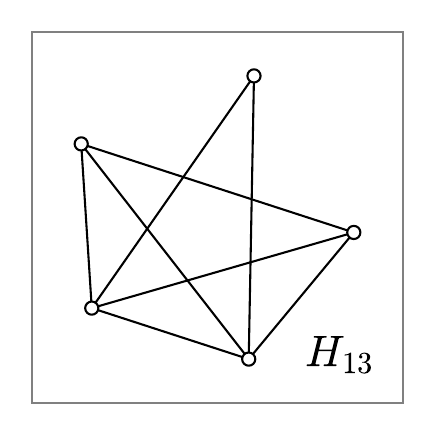} & 
			\includegraphics[width=0.13\textwidth]{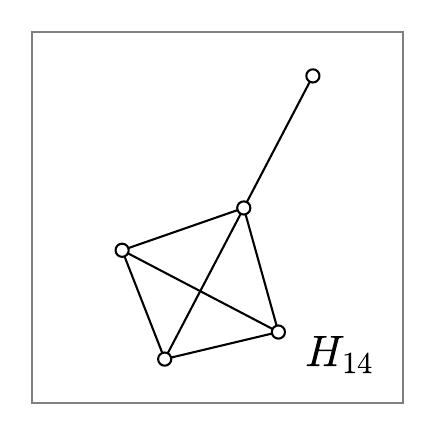} \\
			\includegraphics[width=0.13\textwidth]{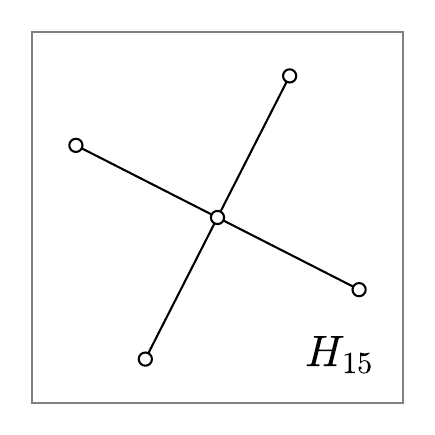} & 
			\includegraphics[width=0.13\textwidth]{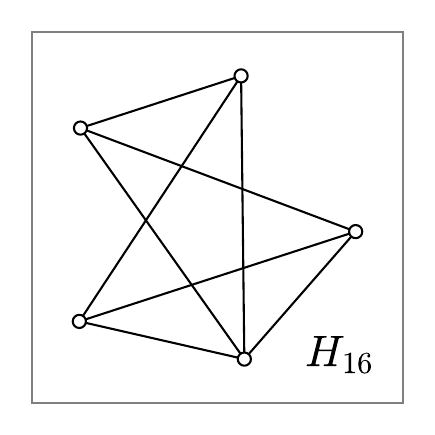} & 
			\includegraphics[width=0.13\textwidth]{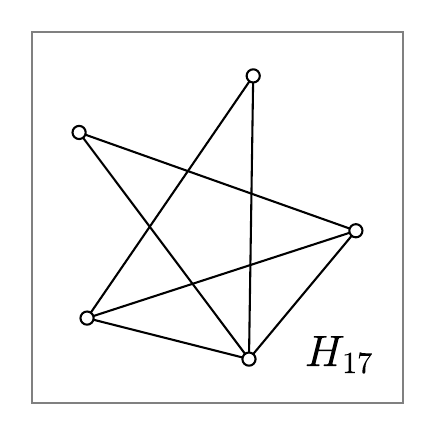} & 
			\includegraphics[width=0.13\textwidth]{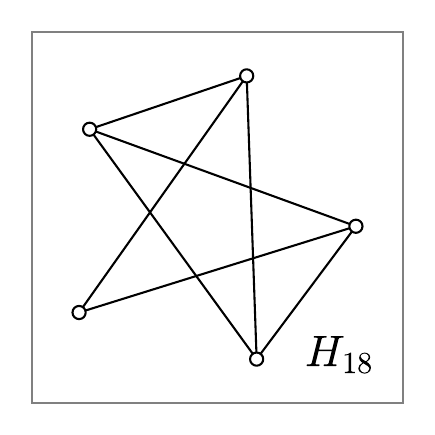} & 
			\includegraphics[width=0.13\textwidth]{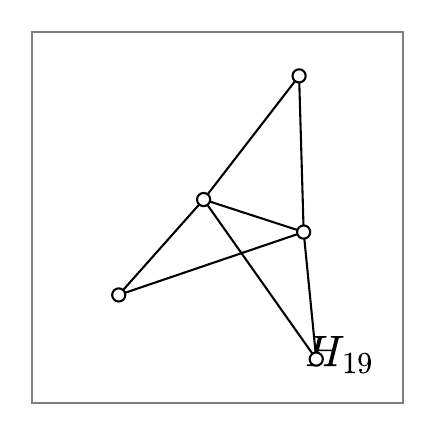} & 
			\includegraphics[width=0.13\textwidth]{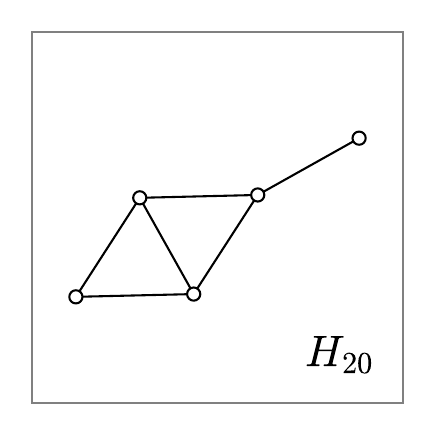} & 
			\includegraphics[width=0.13\textwidth]{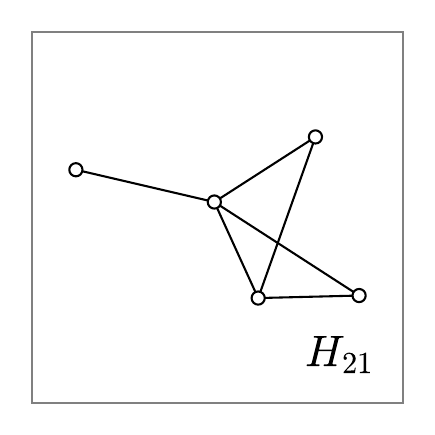} \\				
			\includegraphics[width=0.13\textwidth]{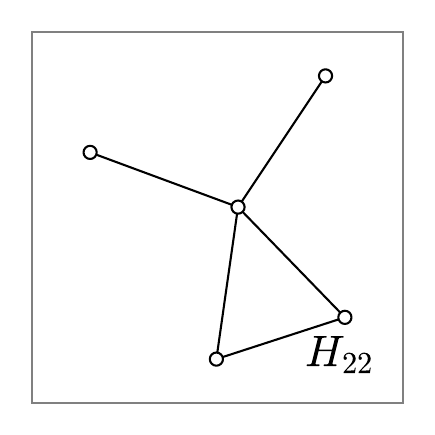} & 
			\includegraphics[width=0.13\textwidth]{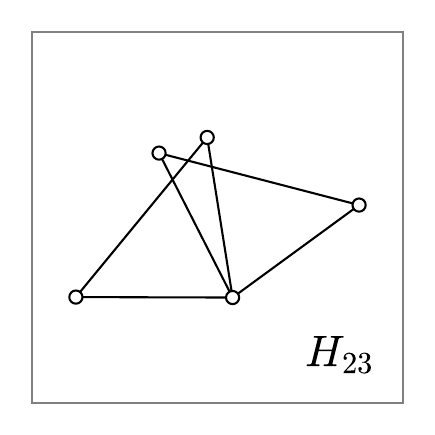} & 
			\includegraphics[width=0.13\textwidth]{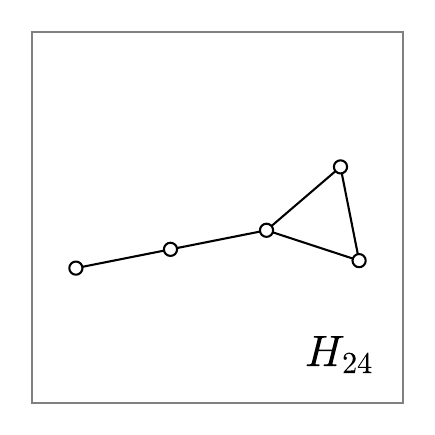} & 
			\includegraphics[width=0.13\textwidth]{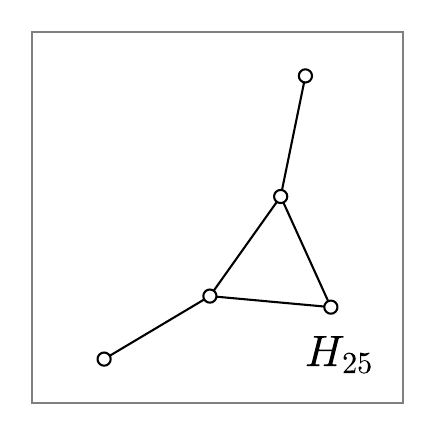} & 
			\includegraphics[width=0.13\textwidth]{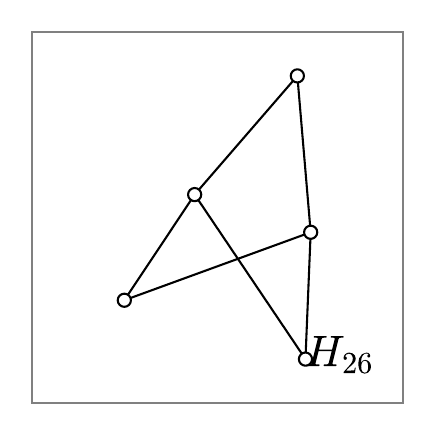} & 
			\includegraphics[width=0.13\textwidth]{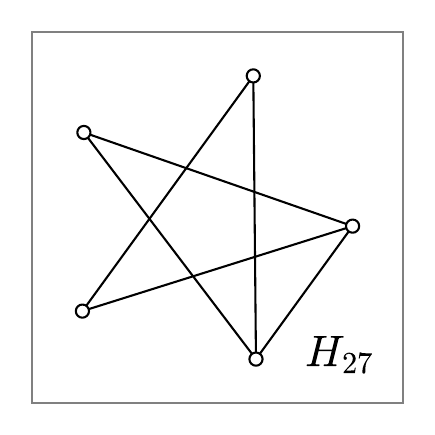} & 
			\includegraphics[width=0.13\textwidth]{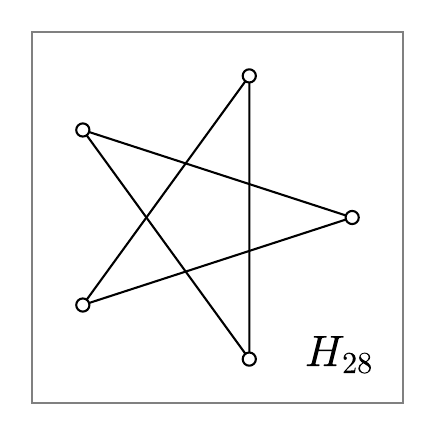} \\
		\end{tabular}
	\end{table}

\begin{proof}
	The theorem is established by the exhaustive application of the inclusion--exclusion formula \eqref{sbh} from Lemma \ref{ide} over all isomorphism classes of graphs up to order 5. Due to the complementation duality proved in Lemma \ref{sym}, it strictly suffices to evaluate the specific canonical vectors $\beta$ listed in Table \ref{efsb}.
	
Rather than detailing the routine algebraic evaluations for all entries, we provide a representative calculation for the coefficient  $s_{\beta}(H_{10})$ with $\beta=(1,1,0,1,0)$,   which is claimed to be $4$ (see Table \ref{efsb}). 

Note that $H_{10}$ is the path graph $P_4$. Since all possible induced subgraphs of $P_4$ constitute a  set 
\begin{equation*}
	\mathcal{L}=\{K_1, K_2, \overline{K_2}, K_1\cup K_2, P_3, P_4\}
\end{equation*} 
we find that
	\begin{equation}\label{p4c}
	s_{\beta}(P_4) = \sum_{L\in 
	\mathcal{L}} (-1)^{4-|V(L)|} \tr\left(\prod_{i=1}^k M_{\beta_i}(L)\right) \ind(L, P_4).
\end{equation}
Since $\beta$ is not the constant vector, we see that the trace expression equals 0 if $L$ or $\overline{L}$ is the edgeless graph. Note that $\ind(P_4,P_4)=1$ and $\ind(K_1\cup K_2,P_4)=\ind(P_3,P_4)=2$. Now, using some direct calculation, we obtain from Eq.~\eqref{p4c} that
$$
	s_\beta(P_4)=\tr\left(\prod_{i=1}^k M_{\beta_i}(P_4)\right)-2\tr\left(\prod_{i=1}^k M_{\beta_i}(P_3)\right)-2\tr\left(\prod_{i=1}^k M_{\beta_i}(K_1\cup K_2)\right)
	=4-0-0=4.
$$
All other coefficients can be computed by the same inclusion--exclusion 
procedure; the full computation for Table \ref{efsb} is automated in \cite{wanggithub2026}.  
\end{proof}
These explicit expansions provide the computational input for the supporter method
developed in the next section. In particular, some graph classes can already be certified
by a single term $n_{\beta}(G)$, while others require suitable linear combinations  of the form $\sum a_{\beta}n_\beta(G)$.
\section{Generalized spectral closedness}
For two (unlabeled) graphs $F$ and $G$, we write $F \preceq G$ if $F$ is an induced subgraph of $G$. This defines a partial order on the set of all finite graphs. Two distinct graphs $G$ and $H$ are \emph{incompatible} if neither $G \preceq H$ nor $H \preceq G$. A family $\mathcal{F}$ is an \emph{antichain} if its members are pairwise incompatible. For example, $\{C_4, P_4, 2K_2\}$ is an antichain, whereas $\{C_4, P_4, P_5\}$ is not, since $P_4 \preceq P_5$. A graph $G \in \mathcal{F}$ is \emph{minimal} if no $H \in \mathcal{F} \setminus \{G\}$ satisfies $H \preceq G$. The set  of minimal elements of $\mathcal{F}$, denoted $\mathcal{F}_{\min}$ is an antichain, and $\Forb(\mathcal{F}) = \Forb(\mathcal{F}_{\min})$. Thus, in the study of graph classes of the form $\Forb(\mathcal{F})$, it is customary to assume that the associated family $\mathcal{F}$ is an antichain. Moreover, to avoid triviality, we usually assume each graph in $\mathcal{F}$ has at least 3 vertices.   

\begin{definition}\normalfont\label{sp}
Let $\mathcal{F}$ be an antichain of graphs and let $\ell\ge \max\{|V(F)|\colon\, F\in \mathcal{F}\}$. A linear combination
\begin{equation*}
	\Phi(G)=\sum_{H\in \mathcal{U}_{\le \ell}}d_H\ind(H,G)
\end{equation*} 
is called an  \emph{$\mathcal{F}$-supporter of order $\ell$} if the coefficients $d_H$ satisfy

(i) $d_H\ge 0$ for $H\in \mathcal{U}_{\le \ell}$;

(ii) $d_H>0$ for $H\in \mathcal{F}$; and

(iii)  $d_H=0$ for $H\in \Forb(\mathcal{F})\cap \mathcal{U}_{\le \ell}$, i.e., $d_H>0$ implies $H$ contains some graph in $\mathcal{F}$ as an induced subgraph. 
\end{definition}
\begin{remark}\normalfont Let $\mathcal{S}=\{H\colon\, H\in \mathcal{U}_{\le \ell} \text{~and~}d_H>0\}$, the supporting set of $\Phi(G)$. Then the last two conditions in Definition \ref{sp} can be equivalently stated as $\mathcal{S}_{\min}=\mathcal{F}$. 
\end{remark}

The importance of Definition \ref{sp} lies in the following simple observation, which gives an equivalent condition for a graph to be $\mathcal{F}$-free.

\begin{lemma}\label{supp}
	Let
	\begin{equation*}
		\Phi(G)=\sum_{H\in \mathcal{U}_{\le \ell}}d_H\ind(H,G)
	\end{equation*} 
	be an $\mathcal{F}$-supporter for some antichain $\mathcal{F}\subset \mathcal{U}_{\le \ell}$. Then for any graph $X$, we have $X\in \Forb(\mathcal{F})$ if and only if $\Phi(X)=0$.
\end{lemma}

\begin{proof}
	Assume first that $X \in \Forb(\mathcal{F})$. For any $H \in \mathcal{U}_{\le \ell}$ with $d_H > 0$, condition (iii) dictates that $H \notin \Forb(\mathcal{F})$, meaning $H$ contains some $F \in \mathcal{F}$ as an induced subgraph. If $\ind(H, X) > 0$, then $X$ must also contain $F$ as an induced subgraph, contradicting $X \in \Forb(\mathcal{F})$. Thus, $\ind(H, X) = 0$ whenever $d_H > 0$. Combined with condition (i), this yields $\Phi(X) = 0$.
	
	Conversely, suppose $X \notin \Forb(\mathcal{F})$. By definition, there exists some graph $F \in \mathcal{F}$ such that $F \preceq X$. Since $\mathcal{F} \subset \mathcal{U}_{\le \ell}$, we have $F \in \mathcal{U}_{\le \ell}$ and $\ind(F, X) \ge 1$. By condition (ii), $d_F > 0$. Since $d_H \ge 0$ for all $H \in \mathcal{U}_{\le \ell}$, we conclude that $\Phi(X) \ge d_F \ind(F, X) > 0$, which means $\Phi(X) \neq 0$.
\end{proof}
	An $\mathcal{F}$-supporter is only an induced-subgraph detector for the hereditary
	class $\operatorname{Forb}(\mathcal{F})$; it is not a spectral object in general. Indeed,
	such supporters exist trivially whenever $\mathcal{F}$ is an antichain, since
	\[
	\sum_{F\in\mathcal{F}} \operatorname{ind}(F,G)
	\]
	is already an $\mathcal{F}$-supporter. Therefore, the nontrivial point in our method is
	not the existence of an $\mathcal{F}$-supporter itself, but whether such a supporter can
	be realized as a linear combination of the generalized spectral invariants $n_\beta(G)$.
	This additional requirement leads to the notion of a walk-realizable
	$\mathcal{F}$-supporter.
\begin{definition}\normalfont
	Let $\Phi$ be an $\mathcal{F}$-supporter of order $\ell$. We say that $\Phi$ is
	\emph{walk-realizable} if there exist real numbers
	$a_\beta$, $\beta\in \mathcal{B}_{\leq \ell}$, such that
	\[
	\Phi(G)=\sum_{\beta\in \mathcal{B}_{\leq \ell}} a_\beta n_\beta(G)
	\]
	for every graph $G$.
\end{definition} 
Equivalently, using the induced-subgraph expansion
\[
n_\beta(G)
=
\sum_{H\in \mathcal{U}_{\leq \ell}} s_\beta(H)\operatorname{ind}(H,G),
\]
an $\mathcal{F}$-supporter
\[
\Phi(G)=\sum_{H\in \mathcal{U}_{\leq \ell}} d_H\operatorname{ind}(H,G)
\]
is walk-realizable if and only if its coefficient vector satisfies
\[
d_H=
\sum_{\beta\in \mathcal{B}_{\leq \ell}} a_\beta s_\beta(H)
\qquad
\text{for every }H\in \mathcal{U}_{\leq \ell}
\]
for some real numbers $a_\beta$, $\beta\in \mathcal{B}_{\leq \ell}$.

We are now in a position to present the main result of this paper. This theorem establishes a direct algebraic criterion for the generalized spectral closedness of $\mathcal{F}$-free graph classes using mixed trace invariants.

\begin{theorem}\label{main}
		Let $\mathcal{F}$ be an antichain of graphs. If for some $\ell\ge \max\{|V(F)|\colon\, F\in \mathcal{F}\}$, there exists a walk-realizable $\mathcal{F}$-supporter $\Phi$ of order $\ell$, then $\Forb(\mathcal{F})$ is generalized spectrally closed.

\end{theorem}
\begin{proof}
		Let $G\in \Forb(\mathcal{F})$, and let $H$ be generalized cospectral with $G$. Since $\Phi$ is a linear combination of $\beta$-closed walk counts, Lemma \ref{binv} implies $\Phi(H)=\Phi(G)$. Since $\Phi$ is an $\mathcal{F}$-supporter, Lemma \ref{supp} gives $\Phi(G)=0$. Hence $\Phi(H)=0$. Applying Lemma \ref{supp} again, we obtain $H\in\Forb(\mathcal{F})$. Therefore $\Forb(\mathcal{F})$ is generalized spectrally closed.
\end{proof}
We first illustrate Theorem \ref{main} in the simplest situation, where the desired
walk-realizable $\mathcal{F}$-supporter is given by a single invariant $n_{\beta}(G)$ for some $\beta$. These examples
show how the induced-subgraph expansions in Table \ref{efsb} directly translate into spectral
closedness certificates, without invoking the full linear programming machinery.

\begin{corollary}
	The class of cluster graphs is generalized spectrally closed.
\end{corollary}

\begin{proof}
	Cluster graphs are exactly $\operatorname{Forb}(P_3)$. By Table \ref{efsb},
	\[
	n_{(1,1,0)}(G)=2\operatorname{ind}(P_3,G).
	\]
	Thus $n_{(1,1,0)}(G)$ is a walk-realizable $\{P_3\}$-supporter. The conclusion follows
	immediately from Theorem \ref{main}.
\end{proof}
\begin{remark}\normalfont
	We note that any cluster graph is  determined by its spectrum \cite[Proposition 6]{vanDam2003}, which implies that the class of cluster graphs  is actually spectrally closed with respect to the ordinary spectrum.  However, it is impossible to construct a walk-realizable $\mathcal{F}$-supporter for cluster graphs using only ordinary closed walks. Indeed, if 
	\[
	\Phi(G)=\sum_{k\le \ell} a_k n_{\mathbf{1}_k}(G)=\tr\sum_{k=1}^\ell a_kA^k \text{~with~} a_\ell\neq0
	\]
	is an $\{P_3\}$-supporter, then we have
	
	 (i) $\ell\ge 3$ since otherwise the coefficient of $\ind(P_3,G)$ in the induced-subgraph expansion of $\Phi(G)$ would be zero, and consequently, 
	 
	(ii) in the induced-subgraph expansion of $\Phi(G)$, the coefficient of $\ind(K_\ell, G)$ receives contribution only from the term $a_\ell \tr(A^\ell)$. Since any sequence of $\ell$ distinct vertices forms a surjective closed walk on the complete graph $K_\ell$, this coefficient is exactly $a_\ell \times \ell!$, which is nonzero as $a_\ell \neq 0$.
	However, $K_\ell$ is a cluster graph (it is $P_3$-free). By Definition~\ref{sp}(iii) of an $\mathcal{F}$-supporter, the coefficient of any $P_3$-free graph must be zero. This yields a contradiction.
\end{remark}
\begin{proof}[Proof of Theorem~\ref{tgsc}]
	Recall that threshold graphs are exactly
	\[
	\operatorname{Forb}(C_4,2K_2,P_4).
	\]
	By Table \ref{efsb}, for $\beta=(1,0,1,0)$ we have
	\[
	n_{\beta}(G)
	=
	8\operatorname{ind}(C_4,G)
	+8\operatorname{ind}(2K_2,G)
	+4\operatorname{ind}(P_4,G).
	\]
	Hence $n_{\beta}(G)$ is a walk-realizable
	$\{C_4,2K_2,P_4\}$-supporter. By Theorem \ref{main}, the class of threshold graphs is
	generalized spectrally closed.
\end{proof}
The two examples above are particularly simple because the single term $n_\beta(G)$ is already
the required supporter. In general, however, no individual $n_\beta(G)$ need have the
correct induced-subgraph support or the required nonnegativity pattern. We therefore
search for a supporter in the linear span of all mixed traces up to a prescribed length.
This leads naturally to a finite linear programming feasibility problem.

Let $\ell\geq \max\{|V(F)|:F\in\mathcal{F}\}$. Since the induced-subgraph expansion
coefficients $s_\beta(H)$ for canonical bracelets can be explicitly computed, searching
for a walk-realizable $\mathcal{F}$-supporter of order $\ell$ is equivalent to finding real
variables $a_\beta$, $\beta\in\mathcal{B}_{\leq \ell}$, such that, with
\[
d_H=
\sum_{\beta\in\mathcal{B}_{\leq \ell}} a_\beta s_\beta(H),
\]
the following constraints hold:
\[
d_H=0
\qquad
\text{for all }
H\in \operatorname{Forb}(\mathcal{F})\cap\mathcal{U}_{\leq \ell},
\]
\[
d_H\geq 0
\qquad
\text{for all }
H\in \mathcal{U}_{\leq \ell}\setminus\operatorname{Forb}(\mathcal{F}),
\]
and
\[
d_H\geq 1
\qquad
\text{for all }
H\in\mathcal{F}.
\]
Since this is a finite linear system with rational coefficients, feasibility
over $\mathbb{R}$ implies the existence of a rational feasible point. In
practice, the numerical output of an LP solver can therefore be converted into,
and then verified as, an exact rational certificate. The code for 
generating and solving these LP systems, as well as for verifying the 
resulting certificates, is publicly available at~\cite{wanggithub2026}.

The class of chain graphs provides the first example in this paper where we use a
nontrivial linear combination of mixed traces rather than a single mixed trace. The LP
formulation above produces the following exact certificate.
\begin{proof}[Proof of Theorem~\ref{cgsc}]
	Recall that chain graphs are exactly
	\[
	\operatorname{Forb}(2K_2,C_3,C_5).
	\]
	Let $\mathcal{F}=\{2K_2,C_3,C_5\}$ and set $\ell=5$. Solving the LP above over
	$\mathcal{B}_{\leq 5}$ yields the following mixed-trace combination:
\begin{equation}
	\label{supchain}
	\begin{aligned}
		\Phi(G)
		={}&
		\frac{1}{8}n_{(1,1)}(G)
		-\frac{1}{12}n_{(1,1,1)}(G)
		+\frac{1}{4}n_{(1,1,0)}(G)
		-\frac{1}{8}n_{(1,1,1,1)}(G)\\
		&-\frac{1}{4}n_{(1,1,1,0)}(G)
		+\frac{1}{8}n_{(1,0,1,0)}(G)
		+\frac{1}{10}n_{(1,1,1,1,1)}(G).
	\end{aligned}
\end{equation}
	Using the expansions in Table \ref{efsb}, this becomes
	\[
	\begin{aligned}
		\Phi(G)
		={}&
		\overline{H_9}+H_3+H_{28}\\
		&+
		9H_5+2H_6+12H_{11}+6H_{12}+2H_{13}
		+4H_{16}+H_{17}+2H_{18}+H_{27},
	\end{aligned}
	\]
where $H_i$ and $\overline{H_i}$ denote $\operatorname{ind}(H_i,G)$ and
$\operatorname{ind}(\overline{H_i},G)$, respectively. All coefficients in this induced-subgraph expansion are nonnegative. The graphs $\overline{H_9},H_3,H_{28}$ are
	respectively $2K_2,C_3,C_5$, and every remaining graph appearing with positive
	coefficient contains one of $2K_2,C_3,C_5$ as an induced subgraph. Hence $\Phi$ is a
	walk-realizable $\mathcal{F}$-supporter. Theorem \ref{main} implies that
	$\operatorname{Forb}(2K_2,C_3,C_5)$ is generalized spectrally closed.
\end{proof}
The computational approach demonstrated above provides a systematic methodology for
testing generalized spectral closedness of $\mathcal{F}$-free graph classes. We applied this framework to the fundamental graph classes listed in Table \ref{fc}.  We utilize standard LP solvers to search for a walk-realizable $\mathcal{F}$-supporter. If such a supporter is successfully found (as demonstrated for cluster, threshold, and chain graphs), Theorem \ref{main} immediately guarantees that the corresponding class is generalized spectrally closed. If the LP system turns out to be infeasible, we turn to a computational search for counterexamples---specifically, pairs of generalized cospectral graphs where one belongs to $\Forb(\mathcal{F})$ and the other does not. The discovery of such a pair definitively proves the class is not generalized spectrally closed, as illustrated by the counterexamples for claw-free, trivially perfect, cograph, and bull-free graph classes (see Figures \ref{claw}--\ref{bull}).

\begin{table}[htbp]
	\centering
	\caption{Generalized spectral closedness of some fundamental graph classes.}
	\label{fcf}
	\begin{tabular}{@{}lcc@{}}
		\toprule
		\textbf{Graph class}  & \textbf{Generalized spectral closedness} & \textbf{Certificate / Counterexample}\\
		\midrule
		Cluster graphs        & Yes &  $n_{(1,1,0)}(G)$\\
		Threshold graphs      & Yes &  $n_{(1,0,1,0)}(G)$ \\
		Chain graphs          & Yes &  Eq.~\eqref{supchain} \\
		\midrule
		Claw-free graphs      & No  & Fig.~\ref{claw} \\
		Cographs              & No  &  Fig.~\ref{cog} \cite{wang2025} \\
		Bull-free graphs      & No  &  Fig.~\ref{bull} \\
		Trivially perfect graphs & No & Fig.~\ref{cog}\\
		\midrule
		Split graphs          & Unknown & Not available\\
		\bottomrule
	\end{tabular}
\end{table}
	\begin{figure}
	\centering
	\includegraphics[height=2.5cm]{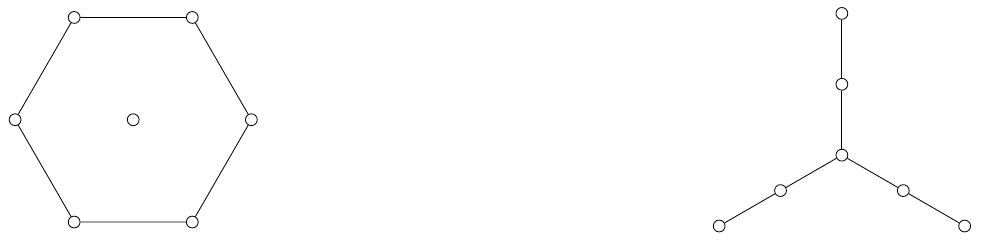}
	\caption{A claw-free graph (left) and its generalized cospectral mate  with a claw (right).}
	\label{claw}
\end{figure}
\begin{figure}
	\centering
	\includegraphics[height=2.5cm]{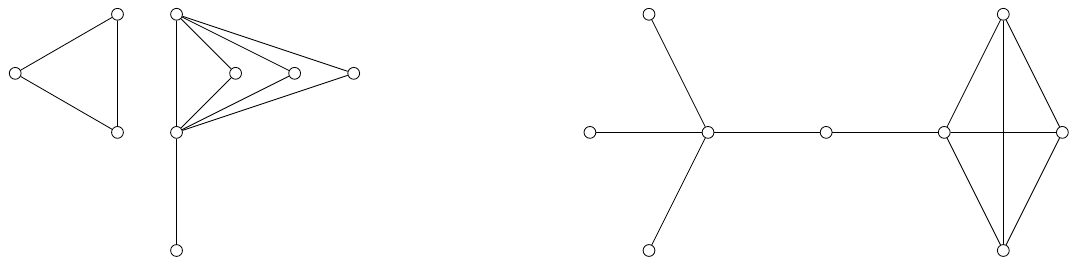}
	\caption{A $\{P_4,C_4\}$-free graph (left) and its generalized cospectral mate   with a $P_4$ (right).}
	\label{cog}
\end{figure}

\begin{figure}
	\centering
	\includegraphics[height=2cm]{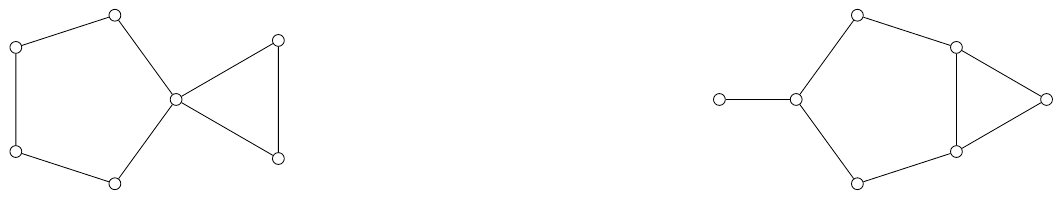}
	\caption{A bull-free graph (left) and its generalized cospectral mate with a bull (right).}
	\label{bull}
\end{figure}

For split graphs, the LP solver indicates that there are no  walk-realizable $\mathcal{F}$-supporters up to order $\ell=7$. Simultaneously, no generalized cospectral counterexamples have been found through an exhaustive computational search over small graphs. Recalling the situation with cluster graphs, where ordinary closed walk counts lacked the degrees of freedom to provide a certificate despite the class being spectrally closed, these computations suggest that the split graph case may require ideas beyond the finite-order certificates found above. This motivates the following conjecture.
 
\begin{conjecture}
Let $\mathcal{F}=\{2K_2,C_4,C_5\}$. The class $\operatorname{Forb}(\mathcal{F})$ of split graphs is generalized spectrally closed, but it admits no walk-realizable $\mathcal{F}$-supporter of order $\ell$ for any integer $\ell \ge 5$.
\end{conjecture}

\section*{Declaration of competing interest}
There is no conflict of interest.

\section*{Acknowledgments and AI disclosure}
W.~Wang is partially supported by the National Natural Science Foundation of China (Grant No.~12001006) and the Wuhu Science and Technology Project, China (Grant No.~2024kj015). Q.~Tang is partially supported by the National Key Research and Development Program of China (Grant No.~2023YFA1010203). 

During the preparation of this work, the authors used AI systems as auxiliary tools. Specifically, GPT-5.5 Pro was used to assist with the early formulation and presentation of the walk-realizable supporter construction for threshold graphs. Additionally, Gemini 3.1 Pro was used to facilitate literature searching, assist with algorithmic programming, and support linguistic refinement. All AI-generated outputs and suggestions were checked, critically evaluated, and restructured by the authors. The authors maintain complete oversight and assume full intellectual and legal responsibility for the integrity, accuracy, and final content of this paper.

\end{document}